\documentclass[a4paper,leqno,english,10pt]{article}
\usepackage[totalheight=26 true cm, totalwidth=17 true cm,]{geometry}

\usepackage{color}
\usepackage{xcolor}
\usepackage{amsthm,amsmath,amssymb,amsfonts,latexsym,url}
\usepackage{bm,enumitem}
\usepackage[colorlinks=false]{hyperref}

\usepackage[utf8]{inputenc}
\usepackage[english]{babel}

\newtheorem{thm}{Theorem}[section]
\newtheorem{cor}[thm]{Corollary}
\newtheorem{lemma}[thm]{Lemma}

\theoremstyle{definition}
\newtheorem{defn}[thm]{Definition}

\theoremstyle{remark}
\newtheorem{rmk}[thm]{Remark}
\newtheorem{exa}[thm]{Example}

\numberwithin{equation}{section}

\def\$$endproof{\eqno{\qedhere}$$\end{proof}}

\def\cL{\mathcal{L}}
\def\cR{\mathcal{R}}

\def\C{\mathbb{C}}
\def\F{\mathbb{F}}
\def\K{\mathbb{K}}
\def\Q{\mathbb{Q}}
     
\def\Z{\mathbb{Z}}

\def\l{\left}
\def\r{\right}
\def\wtilde{\widetilde}

\def\a{\alpha}
\def\be{\beta}
\def\GL{{\rm GL}}
\def\ord{{\rm ord}}

\usepackage[textwidth=2cm]{todonotes}

\title{A Galoisian proof of Ritt theorem on the differential transcendence of Poincaré functions}

\author{
Lucia Di Vizio\footnote{Lucia Di Vizio, CNRS, Université Paris-Saclay, UVSQ, Laboratoire de mathématiques de Versailles, 78000, Versailles, France.
{\tt lucia.di.vizio@math.cnrs.fr}, \url{https://divizio.perso.math.cnrs.fr}}~
\&
Gwladys Fernandes\footnote{Gwladys Fernandes, Université Paris-Saclay, UVSQ, CNRS, Laboratoire de mathématiques de Versailles, 78000, Versailles, France.
{\tt gwladys.fernandes@uvsq.fr}, \url{https://fernandes.perso.math.cnrs.fr}}}

\begin{document}
\bibliographystyle{amsalpha}

\maketitle

\begin{abstract}
Using Galois theory of functional equations, we give a new proof of the main result of the paper ``Transcendental transcendency of certain functions of Poincaré''
 by J.F.~Ritt, on the differential transcendence of the solutions of the functional equation
 $R(y(t))=y(qt)$, where $R(t)\in\C(t)$ verifies $R(0)=0$, $R'(0)=q\in\C$, with $|q|>1$.
We also give a partial result in the case of an algebraic function $R$.
\end{abstract}

\section{Introduction}

\let\thefootnote\relax\footnote{This project has received funding from the ANR project
\href{https://specfun.inria.fr/chyzak/DeRerumNatura/}{DeRerumNatura,
ANR-19-CE40-0018}.}

We fix a rational function $R$, with complex coefficients, such that $R(0)=0$, $R'(0)=q\in\C$, with $|q|>1$.
To linearize the rational map $R$ at $0$, one has to solve the functional equation
\begin{equation}\label{eq:MainFunctEq}
R(\sigma(t))=\sigma(qt).
\end{equation}
It means that, up to a conjugation by $\sigma$, the rational function $R$ acts linearly in the neighborhood of $0$.
H.~Poincaré noticed that such an equation admits a formal solution $\sigma\in t\C[[t]]$, called a Poincaré function,
which is actually the expansion at zero of a uniform function, i.e. a meromorphic function over the whole $\C$. See \cite[page 318]{poincare-1890}.
The functional equation~\eqref{eq:MainFunctEq} plays a key role in rational dynamics and for this reason
has been studied by many authors, in particular relaxing the assumption on the absolute value of $q$.
\par
In \cite{ritt1926transcendental}, J.F.~Ritt addresses the question of the differential algebraicity of $\sigma$ over the field of rational functions
$\C(t)$, i.e., the fact that $\sigma$ is solution of an algebraic differential equation with coefficients in $\C(t)$.
This means that there exists a non-negative integer $n$ and a non-zero polynomial $P\in\C[t,X_0,X_1,\dots,X_n]$ such that
$P(t,\sigma,\sigma',\dots,\sigma^{(n)})=0$. We say that $\sigma$ is differentially transcendental over $\C(t)$ if it is not differentially algebraic.
Ritt's paper is at the origin of a large literature on differential transcendence of solutions of functional equations
linked to dynamical systems, which is surveyed in \cite{Fernandes-Dracula}.

\begin{thm}[{\cite{ritt1926transcendental}}]
\label{thm:Ritt}
Let $\sigma$ be a solution of a functional equation of the form~\eqref{eq:MainFunctEq},
{where $R(t)\in\C(t)$ is not a homography and $R(0)=0$, $R'(0)=q\in\C$, with $|q|>1$.}
If $\sigma$ satisfies an algebraic differential equation then it is
the composition of a homography with a periodic function belonging to the following list:
$\exp(\a t^p)$, $\cos(\a t^p+\be)$,
$\wp(\a t^p+\be)$, $\wp^2(\a t^p+\be)$, $\wp^3(\a t^p+\be)$ or $\wp'(\a t^p+\be)$, where $\wp$ is the Weierstrass function and
for a convenient choice of $p\in\Q$, of $\a\in\C$ and of a fraction $\be\in\C$ of the period.
\end{thm}

The list above appears in \cite{Ritt-PeriodicFncts},
where Ritt classifies periodic Poincaré functions. For more comments and explanations on such a list we refer to
\cite[\S3.1]{Fernandes-Dracula}.
{Notice that $\sigma$ is rational if and only if $R(t)$ is a homography (see Lemma~\ref{lemma:rationalsolutions} below).}
\par
The original proof of the theorem above is organised in the following way:
\begin{enumerate}
  \item Ritt supposes that $\sigma$ is solution of a general algebraic differential equation. Replacing $t$ by $qt$ and using
 ~\eqref{eq:MainFunctEq}, he derives a new differential equation that allows him to perform an Euclidean division in order to lower what he calls the
  \emph{rank} of the differential equation. He iterates this operation several times,
  choosing carefully the terms to eliminate in the process:
  this part of the proof (namely from \S4 to \S11 in \emph{loc.cit.}) is really a \emph{tour de force}.
  He narrows the investigation to three possible
  differential equations satisfied by $\sigma$, that are called $(A)$, $(B)$ and $(C)$ (see \S 11 in \emph{loc.cit.}), and he notices that
  the solutions of $(A)$ and $(B)$ are actually solutions of a differential equation of type $(C)$, so that he is left with this last case.
  \item He shows that the Poincaré functions that are also solutions of $(C)$,
  are obtained by composing a periodic Poincaré function with a rational power of $t$.
  \item He uses his results in \cite{Ritt-PeriodicFncts} on the classification of periodic Poincaré functions to make the list explicit and
  concludes a posterori that actually only the differential equation,
  that he calls $(A)$, can occur:
    $$
    (y')^r=t^jA(y),
    \hbox{~where $r,j\in\Z$ and $A(t)\in\C(t)$.}
    \leqno{(A)}
    $$
\end{enumerate}

Ritt's proof of Theorem~\ref{thm:Ritt} is extremely technical and it may be difficult to understand how
the differential equation $(A)$
is singled out at the very last line of the paper. We think that the Galoisian proof presented below,
which is much shorter modulo the Galois theory, may give a different insight on the existence of the differential equation $(A)$.
The price to pay for a theoretical understanding of the existence of the equation $(A)$ is the algorithmic nature of Ritt's proof.
{Another proof of Ritt theorem can be found in \cite{Casale,casale2015grupoide}, where the same three differential equations as in Ritt's paper appear.}

\bigskip
The purpose of this paper is to give a Galoisian proof of the following theorem,
from whom one can deduce Ritt's theorem (see \S\ref{sec:ritt} below):

\begin{thm}\label{thm:typeA}
{Let $\sigma\in\C[[t]]$ be a formal} solution of a functional equation of the form~\eqref{eq:MainFunctEq},
for some $R(t)\in\C(t)$ not a homography, such that $R(0)=0$, $R'(0)=q\neq 0$ not a root of unity.
The function $\sigma$ is differentially algebraic over $\C(t)$ if and only if
it is solution of a differential equation ``of Ritt type $(A)$'', with $r\neq 0$.
\end{thm}

We close the paper by proving a generalisation of the last statement to the case of
a series $R(t)\in\C[[t]]$ algebraic over $\C(t)$, such that $R(0)=0$ and $R'(0)=q$, with $q\in\C$, $q\neq 0$ not a root of unity.

\medskip
The paper is organized as follows. In \S\ref{sec:ritt} we show how to deduce Theorem~\ref{thm:Ritt} from Theorem~\ref{thm:typeA}.
In \S\ref{sec:Galois} we recall some basic facts of difference Galois theory,
that are used in \S\ref{sec:typeA} to prove Theorem~\ref{thm:typeA}.
Finally, in \S\ref{sec:algebraicR}, we generalize Theorem~\ref{thm:typeA} to the case of an algebraic function $R$.

\paragraph*{Acknowledgement.} It is a pleasure to thank the participants of the \emph{Groupe de travail sur les marches dans le quart de plan},
where both authors have
given several talks on Ritt's theorem and on some subsequent works in differential algebra. We would like to thank Guy Casale and Federico Pellarin for their interest
for the present work and Alin Bostan
for his attentive reading of the manuscript and his remarks that have allowed to improve a previous version of our result.

\section{How to deduce Theorem~\ref{thm:Ritt} from Theorem~\ref{thm:typeA}}
\label{sec:ritt}

The proof below follows the main ideas of Ritt (see \cite[\S12]{ritt1926transcendental}), but it is easier than the
original proof, since we are left from the beginning
with the simplest differential equation among the three found by Ritt.
For the reader convenience we give all the details, even the proofs of the lemmas below that are already in the original reference, and are quite classical.

\smallskip
Let $R(t)\in\C(t)$ be such that $R(0)=0$ and $R'(0)=q\in \C$, with $|q|>1$.
We consider the only solution $\sigma\in t+t^2\C[[t]]$ of the functional equation~\eqref{eq:MainFunctEq}, namely:
\[
\sigma(qt)=R(\sigma(t)).
\]
As we have already pointed out in the introduction, $\sigma$ is the expansion at zero of a uniform meromorphic function over the whole $\C$.

\begin{lemma}[{\cite[\S1]{ritt1926transcendental}}]
\label{lemma:rationalsolutions}
The following assertions are equivalent:
\begin{enumerate}
    \item The Poincaré function $\sigma$ is rational.
    \item $\sigma$ is a homography.
    \item $R(t)$ is a homography.
\end{enumerate}
\end{lemma}

\begin{proof}
If $R(t)$ is a homography, it must have the form $R(t)=\frac{qt}{at+1}$, for some $a\in\C$.
Then $\sigma(t)=\frac{(q-1)t}{at+(q-1)}$.
We conclude proving by contradiction that, if $\sigma$ is rational,
then $R(t)$ is a homography. First, notice that Equation \eqref{eq:MainFunctEq}
implies that:
\begin{equation}
\label{eq:schroder_R2}
    \sigma(q^2t)=R(R(\sigma(t))).
\end{equation}
We suppose that $\sigma\in\C(t)$ and assume that the degree of the numerator of $R$, {after eliminating any common factor with the denominator, is at least $2$.}
In this case, $R(t)$ has at least two finite zeros, namely
$0$, which is simple zero by assumption, and $a_1\neq 0$.
The rationality of $R$ implies that there exists
$a_2$ such that $R(a_2)=a_1$. Note that $a_2\notin\{0,a_1\}$
and $R(R(a_2))=0$.
The rationality of $\sigma$ implies that there exists $b_i$ such that
$\sigma(b_i)=a_i$, for $i\in\{1,2\}$.
Since $b_1\neq b_2$, $b_1$ and $b_2$ cannot be both equal to $\infty$.
If $b_1$ is finite, we deduce from~\eqref{eq:MainFunctEq} that
$\sigma(qb_1)=R(\sigma(b_1))=0$. Recursively we obtain that $\sigma(q^nb_1)=0$
for any positive integer $n$, and hence that $\sigma$ has an infinite number of zeros,
which contradicts its rationality. If $b_2$ is finite, we use \eqref{eq:schroder_R2} to conclude that $\sigma(q^{2n}b_2)=0$
for any positive integer $n$, and hence that $\sigma$ has an infinite number of zeros,
which contradicts its rationality.
Therefore the numerator of $R$ cannot have degree greater than 1, and it is equal to $qt$.
If the denominator of $R(t)$ is a polynomial of degree at least $2$, then
the numerator of $R(R(t))$ has degree at least 2. Applying the previous
reasoning to $R\circ R$ and Equation \eqref{eq:schroder_R2}, we obtain a contradiction.
So both the numerator and the denominator of $R$ have degree at most $1$, which means that $R$ is a homography.
\end{proof}

\begin{cor}[{\cite[\S1 to \S3]{ritt1926transcendental}}]
If $R(t)$ is not a homography, then:
\begin{enumerate}
    \item $\sigma$ has an infinite set of zeros and an essential singularity at $\infty$.
    \item $\sigma$ is transcendental.
\end{enumerate}
\end{cor}

\begin{proof}
Let us suppose that the numerator and the denominator of $R(t)$ do not have
common factors.
If $R$ is not a homography, then either the numerator or the denominator of $R$ has degree at least $2$.
As in the proof of the previous lemma, we conclude that
$\sigma$
has an infinite set of zeros, that accumulate at $\infty$. Since
$\sigma$ is not identically zero, it must have an essential singularity
at $\infty$. This also implies that $\sigma$ is transcendental.
\end{proof}

\begin{proof}[Proof of Theorem~\ref{thm:Ritt}]
{By Theorem \ref{thm:typeA}}, we have $(\sigma')^r=t^jA(\sigma)$, for some $r,j\in\Z$,
with $r\neq 0$, and $A(t)\in\C(t)$, and hence:
\[
0=\ord_{t=0}(\sigma')^r=j+\ord_{t=0}A(t)\,\ord_{t=0}\sigma(t).
\]
We conclude that $\ord_{t=0}A(t)=-j$.
We know from the previous corollary that $\sigma$ has a finite zero $b\neq 0$.
Let $p:=\ord_{t=b}\sigma(t)\geq 1$.
We observe that:
\[
r\,\ord_{t=b}\sigma'(t)=\ord_{t=0}A(t)\,\ord_{t=b}\sigma(t)=-jp,
\]
hence $r(p-1)+jp=0$. We consider the change of variable $t=u^p$ and we set
$z(u):=\sigma(u^p)$. A direct calculation shows that:
    \begin{equation}\label{eq:monodromy}
    z(q^{1/p}u)=R(z(u))
    \hbox{~and~}
    z'(u)^r=p^rA(z(u)).
    \end{equation}
Notice that for any non-zero constant $\wtilde u$, the function
$z(u+\wtilde u)$ is also a solution of the differential equation above.
Since $\sigma$ has an essential singularity at $\infty$,
then $\sigma$ takes all the values in $\C$, with at most one exception.
The same holds for $z(u)$. Therefore the unique solution of the
differential equation $y'(u)^r=p^rA(y(u))$ with initial condition
$y(u_0)=c$, with $u_0,c\in\C$, is constructed in the following way: we find $u_c$
such that $z(u_c)=c$ and we chose the solution $z(u-u_0+u_c)$ of the
differential equation above.
Given the freedom in the choice of both $u_0$ and $c$,
one avoids the missing value of $z(u)$ and concludes that all solutions of  $y'(u)^r=p^rA(y(u))$ are obtained composing
$z(u)$ with a translation.
\par
By construction, the solution $z(u)$ has a non-trivial monodromy.
Let $\wtilde z(u)$ be another branch of $z(u)$.
We observe that $\wtilde z(u)$ is solution of the
system~\eqref{eq:monodromy}, therefore $\wtilde{z}(u)=z(u+\wtilde u)$, for some
non-zero $\wtilde u\in\C$.
The uniqueness of the analytic continuation implies that
$\wtilde z(q^{1/p}u)=R(\wtilde z(u))$ and hence that:
\[
z(q^{1/p}u+\wtilde u)=\wtilde z(q^{1/p}u)=R(\wtilde z(u))
=R(z(u+\wtilde u))
=z(q^{1/p}(u+\wtilde u)).
\]
We deduce that $z$ is periodic of period $(q^{1/p}-1)\wtilde u$.
We conclude as Ritt does, using his result \cite{Ritt-PeriodicFncts}
on the classification of periodic Poincaré functions.
\end{proof}

\section{Elements of Galois theory of difference equations}
\label{sec:Galois}

For an introduction to the Galois theory of difference equations, we refer the reader to \cite{vdPutSingerDifference}, or to
\cite[\S 2]{OvchinnikovWibmer},
where the authors make very general assumptions.
The setting considered here (as well as the notation) matches the approach
developed in \cite[\S2 to \S5]{divizio-hanoi}, so we are going to
refer to it. The only difference with \cite{divizio-hanoi} is that our field of
constants $C$ is algebraically closed. For this reason we can naively
identify the Galois groups with their $C$-points, which makes things slightly easier.
We remind below the notions that are essential to the understanding of the proof in the next section.

\bigskip
Let $\F/\K$ be a field extension, such that $\F$ comes equipped with an endomorphism $\Phi:\F\to\F$, which induces a
{non-periodic} endomorphism of $\K$.
{It means that there exists $x\in\K$ such that $\Phi^n(x)\neq x$, for any
{non-zero} integer $n$.}
We suppose that the field $C$ of elements of $\F$ that are left invariant by $\Phi$ is algebraically closed and that
$C\subset \K$.
We consider the linear system of the form
\begin{equation}\label{eq:system}
\Phi\vec y=\begin{pmatrix}
a_1 & b\\ 0 & a_2
\end{pmatrix}\vec y,
\end{equation}
where $a_1,a_2,b\in\K$, with $a_1a_2\neq 0$, and
we suppose that there exists an invertible matrix
$\begin{pmatrix}
z_1 & w\\ 0 & z_2
\end{pmatrix}\in\GL_2(\F)$ that satisfies~\eqref{eq:system}.

\begin{defn}[{see \cite[Def.~3.5]{divizio-hanoi}}]
We call Picard-Vessiot ring of~\eqref{eq:system} over
$\K$ the ring $\cR=\K[z_1^{\pm 1},z_2^{\pm 1},w]\subset\F$.
We define the Galois group of~\eqref{eq:system} to be:
\begin{equation}\label{eq:GaloisGroup}
    G:=\{\varphi:\cR\to \cR,
    \hbox{~automorphism of $\K$-algebras, such that~}\varphi\circ\Phi=\Phi\circ\varphi\}.
\end{equation}
The elements of $G$ extend to automorphisms of the field of fractions $\cL$ of $\cR$.
\end{defn}

The system~\eqref{eq:system} boils down to the equations $\Phi(z_i)=a_iz_i$, for $i=1,2$, and $\Phi(w)=a_1w+bz_2$.
Any
$\varphi\in G$, being a ring automorphism, must leave globally invariant the space of solutions of \eqref{eq:system}.
Therefore there exists non-zero $c_i\in C$ such that
$\varphi(z_i)=c_iz_i$, for $i=1,2$. As far as $\varphi(w)$ is concerned,
it must be a solution of $\Phi(y)=a_1 y+b\varphi(z_2)=a_1 y+b c_2 z_2$, hence there exists $d\in C$ such that
$\varphi(w)=dz_1+c_2w$.
We conclude that
\[
\varphi\begin{pmatrix}
z_1 & w\\ 0 & z_2
\end{pmatrix}=\begin{pmatrix}
z_1 & w\\ 0 & z_2
\end{pmatrix}\begin{pmatrix}
c_1 & d\\ 0 & c_2
\end{pmatrix}
=\begin{pmatrix}
c_1z_1 & dz_1+c_2w\\ 0 & c_2z_2
\end{pmatrix},
\]
for some $\begin{pmatrix}
c_1 & d\\ 0 & c_2
\end{pmatrix}
\in\GL_2(C)$.

\bigskip
We now state the main properties of the Galois group of a functional equation:

\begin{thm}[{\cite[Thm.~4.9 and 5.3]{divizio-hanoi}}]
\begin{enumerate}
    \item The Galois group $G$ is an algebraic subgroup of
    $\GL_2(C)$, and its dimension as an algebraic variety over $C$ is equal to the transcendence degree of $\cL/\K$.
    \item $\K=\{f\in \cL: \varphi(f)=f~\forall \varphi\in G\}$.
\end{enumerate}
\end{thm}

\begin{exa}\label{exa:Gm2}
Let us consider a special case of the system~\eqref{eq:system}, with $b=0$ and, therefore, $w=0$.
In this case the Picard-Vessiot ring coincides with $\cR:=\K[z_1^{\pm 1},z_2^{\pm 1}]$ and
its Galois group $G$ is a subgroup of the group of the invertible diagonal matrices of rank $2$,
that we can naively identify with $(C^\ast)^2$.
It means that for any automorphism $\varphi\in G$ there exist two non-zero constants
$c_1,c_2\in C$ such that $\varphi(z_i)=c_iz_i$, for $i=1,2$.
\par
The solutions $z_1,z_2$ are algebraically dependent over $\K$ if and only if $G$ has dimensions $0$ or $1$, or equivalently if it
is a proper algebraic subgroup of  $(C^\ast)^2$.
We know that the proper algebraic subgroups of  $(C^\ast)^2$ are defined
by equations of the form $X_1^{\a_1}X_2^{\a_2}=1$, for some
$\a_1,\a_2\in\Z^2\smallsetminus\{(0,0)\}$.
It means that $c_1^{\a_1}c_2^{\a_2}=1$.
We conclude from the theorem above that $z_1^{\a_1}z_2^{\a_2}\in\F$
is invariant under any automorphism of the Galois group $G$ and hence that
$z_1^{\a_1}z_2^{\a_2}\in\K$.
\end{exa}

\begin{exa}\label{exa:Ga}
If in~\eqref{eq:system} we set $a_1=a_2=1$, then we can take $z_1=z_2=1$. The Picard-Vessiot ring boils down to
$\cR=\K[w]$ and the Galois group $G$ is a subgroup of
the group of matrices $\l\{\begin{pmatrix}1& d\\0& 1\end{pmatrix}, d\in C\r\}$,
that we can identify to $(C,+)$.
It means that for every $\varphi\in G$, there exists a constant $d\in C$ such that $\varphi(w)=w+d$.
\par
The solution $w$ is algebraically dependent
over $\K$ if and only if $G$ is a proper algebraic subgroup of $(C,+)$,
but the only proper algebraic subgroup of $(C,+)$ is the trivial group $0$.
Hence, if $w$ is algebraic over $\K$, then $d=0$ and $w$
is tautologically left fixed by all morphisms of $G$.
We conclude that $w\in\K$.
\end{exa}

\section{Proof of Theorem~\ref{thm:typeA}}
\label{sec:typeA}

Let $C$ be an algebraically closed field and $\F:=C((x))$.
We fix a formal power series $R\in\F$ such that $R$ is not a homography, {$R(0)=0$, $R'(0)=q\neq 0$ is not a root of unity}, so that
we can define the morphism:
$$
\begin{array}{cccc}
  \Phi_R: & \F & \to & \F, \\
   & f & \mapsto & f(R(x)).
\end{array}
$$
Notice that $\Phi_R$ is an automorphism of $\F$.
Moreover, $\Phi_R$ is not periodic.

\begin{lemma}\label{lemma:constants}
The field of constants $\F^{\Phi_R}:=\{f\in\F:\Phi_R(f)=f\}$ of $\F$
with respect to $\Phi_R$ coincides with $C$.
\end{lemma}

\begin{proof}
Let $f\in \F\smallsetminus C$ be such that $f(R(x))=f(x)$.
Replacing $f$ with $f-f(0)$, we can suppose that $f$ has no constant term and,
replacing $f$
with its inverse for the Cauchy product, we can suppose that
$f=\sum_{n\geq N}f_nx^n$, for some positive integer $N$ that we chose so that $f_N\neq 0$.
{The coefficient of $x^N$ in $f(R(x))$ is $f_Nq^N$.
We deduce from $f(R(x))=f(x)$ that $f_Nq^N=f_N$, with $q^N\neq 0,1$, and hence that $f_N=0$, against
our assumption.} We conclude that, if $f$ has no constant term, then $f=0$.
This means exactly that $\F^{\Phi_R}=C$.
\end{proof}

Let $\K:=\C(x)$ be the field of rational functions.
We assume that $R$ is (the expansion of) a rational function.
In this case $\Phi_R$ an endomorphism of $\K$ and
we consider the functional equation:
\begin{equation}\label{eq:Konig}
\Phi_R(y_0(x))=qy_0(x).
\end{equation}
We use the notation $y_0$ for the unknown function for reasons that will be clearer in a few lines. The equation above has a formal solution $\tau\in\F$. It satisfies necessarily $\tau(0)=0$. Notice that $\tau$ is determined up to a multiplicative constant.
We have:

\begin{thm}\label{thm: main}
Let $\tau\in\F$ be a formal solution of \eqref{eq:Konig}.
The following assertions are equivalent:
\begin{enumerate}
    \item $\tau$ is differentially algebraic over $C(x)$.
    \item $\tau$ satisfies a differential equation with coefficients in $C(x)$ of order $1$.
    \item $\tau$ is solution of the differential equation
    $(y')^r=A(x)y^j$, for some $(r,j)\in\Z^2$, with $r\neq 0$,
    and $A(x)\in C(x)$.
\end{enumerate}
\end{thm}

\begin{proof}
Notice that the implications $3\Rightarrow 2\Rightarrow 1$ are tautological.
Proving that $1\Rightarrow 3$ would end both the proof of this theorem and of Theorem~\ref{thm:typeA}: to achieve
this purpose we will rather prove that $1\Rightarrow 2\Rightarrow 3$.
\par
We start proving that $2\Rightarrow 3$.
We consider the system of functional equations
\begin{equation}\label{eq:RittDiagonalized1}
\l\{
\begin{array}{l}
\Phi_R(y_0)=qy_0,\\
\displaystyle\Phi_R(y_1)=\frac{q}{R'}y_1.
\end{array}\r.
\end{equation}
We are in the case of Example~\ref{exa:Gm2}.
The fact that $\tau$ is solution of a differential equation of order $1$ means that the system above has a basis of solutions,
namely $\tau,\tau'$, which are algebraically dependent over $\K$.
Following the example, we conclude that
$\tau^{\a_1}(\tau')^{\a_2}\in\K$,
for some integers $\a_1,\a_2$, which cannot be simultaneously zero.
If $\a_2\neq0$, it is enough to
rephrase the latter statement in Ritt's notation: There exist $r,j\in\Z$, with $r\neq 0$, and $A(x)\in\K$ such that $\tau$ is solution of the differential equation $(y')^r=A(x)y^j$, as claimed.
On the other hand, if $\a_2=0$, then $\tau^{\a_1}\in\K$, and we can conclude by taking the logarithmic derivative.
\par
We now prove by contradiction that $1\Rightarrow 2$.
So let us suppose that $\tau$ is solution of an algebraic differential equation
of order $n$ such that $n>1$ is minimal for this property. We set $\K_n:=\K(\tau,\tau',\dots,\tau^{(n-1)})$.
The definition of $n$ implies that $\tau^{(n)}$ is algebraic over $\K_n$, but $\tau^{(n)}\not\in\K_n$, otherwise we would obtain a
differential equation of order $n-1$ for $\tau$.
Notice that deriving $\Phi_R(y_0)=qy_0$, calling $y_k$ the $k$-th derivative of $y$
and using the Faà di Bruno's formula we obtain:
\[
(\Phi_R(y_0))^{(n)}=\sum_{k=1}^nB_{n,k}(R',\dots,R^{(n-k+1)})\Phi_R(y_k)=qy_n,
\]
where $B_{n,k}(x_1,\dots,x_{n-k+1})$ are the Bell polynomials defined by the multivariate identity:
$\exp\l(\sum_{j=1}^{\infty}x_j\frac{t^j}{j!}\r)=\sum_{n\geq k\geq 0}B_{n,k}(x_1,\dots,x_{n-k+1})\frac{t^n}{n!}$.
As $B_{n,n}(x_1)=x_{1}^{n}$,
replacing the $\Phi_R(y_k)$'s recursively up to $y_{n-1}$, we obtain an expression of the form:
\[
\Phi_R(y_n)=\frac{q}{(R')^n}y_n+\sum_{k=1}^{n-1}A_{n,k}(x)y_k,
\]
where $A_{n,k}(x)\in C(x)$ is a rational function, which is actually a rational expression in the derivatives of $R$.
We set $b:= \sum_{k=1}^{n-1}A_{n,k}(x)\tau^{(k)}\in\K_n$
and $z=(\tau')^n\tau^{1-n}\in \K_2\subset\K_n$, to simplify the notation.
We deduce from the functional equation above that $\omega_n:=\frac{\tau^{(n)}}{z}$ verifies the functional equation
\begin{equation}\label{eq:Ritt2}
    \Phi_R(\omega_n)=\omega_n+\frac{b(R')^n}{qz}.
\end{equation}
We are in the situation of Example~\ref{exa:Ga}.
Since $\omega_n$ is algebraic over $K_n$, we conclude that $\omega_n\in\K_n$.
Since $z\in\K_n$, we conclude that also $\tau^{(n)}\in\K_n$,
which is in contradiction with our choice of $n$.
This ends the proof.
\end{proof}

We deduce theorem Theorem~\ref{thm:typeA} from the statement above.

\begin{proof}[Proof of Theorem~\ref{thm:typeA}]
Let $\sigma$ be the inverse of $\tau$ for the composition.
First of all, let us notice that $\sigma$ is differentially algebraic over $C(t)$ if and only if $\tau$ is differentially algebraic over $C(x)$
(see \cite[page 344]{Boshernitzan-Rubel-1986} or \cite[page 55, (n)]{moore1896concerning}).
We consider the local change of variable $t=\tau(x)$
in \eqref{eq:MainFunctEq}, or equivalently $x=\sigma(t)$, which
transforms \eqref{eq:Konig} and $(\tau')^r=A(x)\tau^j$ into \eqref{eq:MainFunctEq} and $(\sigma')^{-r}=t^jA(\sigma)$,
respectively.
\end{proof}

\section{The algebraic case}
\label{sec:algebraicR}

We change a little bit the notation with respect to the previous section.
Let us consider the relative closure $\K$ of $C(x)$ inside $\F:=C((x))$
and $R(x)\in\K$, such that $R$ is not a homography,
$R(0)=0$, $R'(0)=q\neq 0$ {is not a root of unity}.
Then the functional equation
\[
\tau(R(x))=q\tau(x)
\]
has a formal solution at $0$.
\par
We can define
the automorphism $\Phi_R$ of $\F$ as in the previous section. Then $\Phi_R$ induces an automorphism of $\K$ and the field of constants is $C$,
as in Lemma~\ref{lemma:constants}.
Reasoning word by word as in the previous section one can prove:

\begin{thm}
Let $\tau(x)$ be a formal
solution of $\tau(R(x))=q\tau(x)$ in $C[[x]]$.
If $\tau(x)$ is differentially algebraic over $\K$, then it satisfies a differential equation
of the form $(y')^r=y^jA(x)$, where $r,j$ are integers, with $r\neq 0$,
and $A(x)\in\K$.
\end{thm}

\begin{rmk}
The fact that $\K/C(x)$ is algebraic implies that
being differentially algebraic over $\K$ is equivalent to being differentially algebraic over $C(x)$. Since $A(x)\in\K$, there exists $P(x,T)\in C[x,T]$ such that
$P(x,A(x))=0$. Then $P(x,(y')^ry^{-j})=0$ provides a differential equation for
$\tau$ over $C(x)$.
\end{rmk}

\providecommand{\bysame}{\leavevmode\hbox to3em{\hrulefill}\thinspace}
\providecommand{\MR}{\relax\ifhmode\unskip\space\fi MR }
\providecommand{\MRhref}[2]{%
  \href{http://www.ams.org/mathscinet-getitem?mr=#1}{#2}
}
\providecommand{\href}[2]{#2}

\end{document}